\documentclass[review]{elsarticle}

\usepackage{amsmath}
\usepackage{latexsym, amssymb}
\usepackage{graphicx}
\usepackage{amsmath, amsbsy}
\usepackage{amsopn, amstext}
\usepackage{hyperref}
\usepackage{cancel, color}
\usepackage{epstopdf}

\def\J{{\bf 1}}

\DeclareMathOperator{\rank}{rank}

\DeclareMathOperator{\Span}{Span}

\DeclareMathOperator{\argmin}{argmin}
\DeclareMathOperator{\lcm}{lcm}

\def\cal{\mathcal}

\def\ra{\rightarrow}

\def\d{\delta}

\def\l{\lambda}

\def\0{{\bf 0}}

\newcommand{\R}{{\mathbb R}}

\newcommand{\Z}{{\mathbb Z}}

\def\dsum{\mathop{\sum}\limits}

\newtheorem{thm}{Theorem}[section]
\newtheorem{dfn}[thm]{Definition}
\newtheorem{prp}[thm]{Proposition}
\newtheorem{exa}[thm]{Example}
\newtheorem{lem}[thm]{Lemma}
\newtheorem{cor}[thm]{Corollary}
\newtheorem{rem}[thm]{Remark}
\newtheorem{alg}[thm]{Algorithm}

\begin{document}

\begin{frontmatter}

\title{A Numerical Solution to KPD \tnoteref{footnoteinfo}}
% Title, preferably not more

\tnotetext[footnoteinfo]{ This work is supported partly by the National Natural Science Foundation of China (NSFC) under Grant 62350037. }

\author[AMSS]{Daizhan Cheng}\ead{dcheng@iss.ac.cn}
%\thanks[CA]{Corresponding author. Tel. +86-10-6265 1445. Fax +86-10-6258 7343.}
%\author[AMSS,CAS]{Zhengping Ji}\ead{jizhengping@amss.ac.cn}

\address[AMSS]{Key Laboratory of Systems and Control, Academy of Mathematics and Systems Science, Chinese Academy of Sciences, Beijing 100190, P.R.China}
%\address[STPC]{Research Center of Semi-tensor Product of Matrices, Theory and Applications, Liaocheng University, Lianocheng, P.R. China}
%\address[CAS]{School of Mathematical Sciences, University of Chinese Academy of Sciences, Beijing 100049, P.R.China}

\begin{abstract}
A stationary value based algorithm (SVA) is provided to solve the nearest Kronecker product decomposition (KPD) problem of vector form hypermatrices. Using the algorithm successively,  the finite sum KPD is also solved.
Then the permutation matrix is introduced. Using it, the KPD of matrix form hypermatrices is converted to its equivalent KPD of vector forms, and then the SVA is also applicable to solve the same problems for vector form hypermatrix.
Some numerical examples are presented to demonstrate the new algorithm  and to compare it with existing methods.
\end{abstract}


\end{frontmatter}

\section{Introduction}

According to \cite{mar04},  the tensor decomposition was independently proposed by \cite{car70} and \cite{har70}. The problem is described as follows: Let ${\cal A}=(a_{i_1,\cdots,i_d})\in \R^{n_1\times\cdots\times n_d}$ be a hypermatrix, (which is also called a tensor), of order $d$ and dimension $n_1\times\cdots\times n_d$. Express it into a decomposition form as \cite{mar04}
\begin{align}\label{1.1}
{\cal A}=\dsum_{k=1}^r (x^k_1\otimes \cdots\otimes x^k_d),
\end{align}
where
$$
x^k_j\in \R^{n_j},\quad j\in [1,d],~k\in [1,r].
$$
The smallest $r$ is called the tensor rank of ${\cal A}$, denoted by
$$
\rank_t({\cal A}):=\argmin_{r}\{\mbox{decomposition~(\ref{1.1})~holds}\}.
$$

Recently, there has been renewed interest in this decomposition, because it has been found wide applications in many newly appearing large scale and large date systems in networked systems, signal processing, generated artificial intelligence (GAI), etc. It is also commonly called as the Kronecker product decomposition (KPD) \cite{van93}.

The applications of KPD  cover wide range of various fields, for example, (a) various problems from numerical linear algebra, referring to \cite{cha00} and the references therein; (b) analysis and synthesis of logical dynamic systems \cite{wei23}; (c) analysis of medical imaging data\cite{fen24}; (d) image processing
and preconditioning\cite{wu23}; (e) direction-of-arrival estimation in traffic and manufacturing systems \cite{wan21,joe24}; (f) room acoustic impulse responses \cite{dog22};(g) system identification \cite{pal18}, etc.

Particularly, the KPD technique has also been found many applications in AI, because it can
greatly reduce the total number of parameters in the model. This is particularly important for large scale models.  The related researches include (a)
    training block-wise sparse matrices for machine learning \cite{zhu24}; (b) GPT compression \cite{tah21, eda21,ayo24}; (c) CNN compression \cite{ham22}; (d) image process \cite{wu23b}, etc.

The main efforts have been focused on singular value decomposition (SVD) based approximation to the solution of KPD of matrices \cite{cos08,gra10,gar18}. In addition, some direct decomposition methods are also proposed\cite{bat17,wu23}.

There has also been some efforts in KPD of hypermatrices\cite{bat17,wan22}.

The KPD described by (\ref{1.1}) is called the KPD of vector form hypermatrices, because in (\ref{1.1}) ${\cal A}$ is considered as a vector, and all the decomposed components are vectors.

An alternative form KPD is the matrix form one, which can be described as follows: Let $A\in {\cal M}_{m\times n}$, where $m=\prod_{s=1}^dm_s$ and $n=\prod_{s=1}^dn_s$. The matrix form KPD is described as\cite{bat17}
\begin{align}\label{1.2}
A=\dsum_{k=1}^r A^k_1\otimes \cdots\otimes A^k_d,
\end{align}
where $A^s_k\in {\cal M}_{m_k\times n_k}$, $s\in[1,d]$, $k\in[1,r]$.

The KPD related problem is the nearest Kronecker product (NKP) problem, which can be described as follows.  Let $A\in {\cal M}_{m\times n}$, where $m=\prod_{s=1}^dm_s$ and $n=\prod_{s=1}^dn_s$. The NKP is described as follows \cite{van00}: Find $A^s_*$, $s\in [1,d]$ such that
\begin{align}\label{1.3}
A\approx A^1_*\otimes \cdots\otimes A^d_*,
\end{align}
where
$$
(A^1_*,\cdots,A^d_*)=\argmin_{A_1,\cdots,A_d}\|A-A^1\otimes \cdots\otimes A^d\|_F.
$$

It is obvious that (\ref{1.3}) is the matrix form NKP, and its corresponding vector form is as follows. Let ${\cal A}\in \R^{n_1\times \cdots\times n_d}$. Find  $x^s_*\in \R^{s}$, $s\in [1,d]$, such that
\begin{align}\label{1.301}
{\cal A}\approx \ltimes_{s=1}^dx^s_*,
\end{align}
where
$$
(x^1_*,\cdots,x^d_*)=\argmin_{x^1,\cdots,x^d}\|{\cal A}-\ltimes_{s=1}^dx^s\|_F.
$$

It is well known that a minimum sum KPD, i.e., with respect to the tensor rank, is a highly challenging problem\cite{one25}.

An interesting special case is the exact KPD, which asks in vector form when there exist $x^s_*$, $s\in [1,d]$ such that (\ref{1.301}) becomes an equality?  Or in matrix form when there exist $A^s_*$, $s\in [1,d]$ such that (\ref{1.3}) becomes an equality?

Recently, a numerical algorithm, called the  monic decomposition algorithm (MDA), has been developed \cite{chepr}.
Using MDA, an easily verifiable necessary and sufficient condition is obtained \cite{chepr}. There are many fundamental problems about tensor decomposition, which have been listed in \cite{mar04}. Although there have been  many research results in this field,  to our best knowledge,  general answers  to both KPD and NKP are still unknown.

The purpose of this paper is to develop a numerical algorithm, called the SVA to solve NKP. The algorithm aims to find  stationary values of the square error, and then find the minimum one. By using NKP iteratively, we can finally obtain the finite sum KPD. Numerical examples show that the SVA is very efficient and accurate. The advantages and disadvantages of SPBA, comparing with singular value based decomposition, are  discussed.

The rest of this paper is organized as follows.
Section 2 provides some necessary preliminaries, including (a) semi-tensor product (STP) of matrices; (b) hypermatrices;
(c) hypervectors. Section 3 considers the KPD of vector form hypermatrices. First, the exact decomposition is considered, the MDA and MDA based necessary and sufficient condition are reviewed. Then the NKP is considered. As the main result of this paper, an efficient algorithm, called the SVA, is proposed. Finally, using SVA iteratively yields the finite sum KPD. Section 4 introduces the permutation matrix. Using it, we convert the KPD of matrix form hypermatrices to the KPD of its equivalent vector form. A numerical example is presented to demonstrate the theoretical results.
Finally, a simple conclusion is presented in Section 5.

Before ending this section a list of notations is attached as follows.

\begin{enumerate}
\item $\R^n$: $n$ dimensional real Euclidean space.
\item $\Z_+$: Set of positive integers.

\item ${\cal M}_{m\times n}$: the set of $m\times n$ matrices.
%
%\item $\Col_i(A)$: i-th column of $A$.
%
\item $\lcm(a,b)$: least common multiple of $a$ and $b$.
\item $A\ltimes B$: semi-tensor product of $A$ and $B$.

\item  $\R^{n_1\times n_2\times \cdots \times n_d}$: the $d$-th order hypermatrices of dimensions $n_1\times\cdots\times n_d$.
\item  $\R^{n_1\ltimes n_2\ltimes \cdots \ltimes n_d}$: the $d$-th order hypervectors of dimensions $n_1,n_2,\cdots,n_d$.
\item ${\cal A}, {\cal B}$, etc.: hypermatrices.

\item $\vec{i}=Id(i;n)$: index $i\in [1,n]$,

\item $M^{\vec{j}\times \vec{k}}({\cal A})$: matrix expression of ${\cal A}$, which has rows labeled by $\vec{j}$ and columns labeled by $\vec{k}$.

\item $V_r(A)$ ($V_c(A)$): row stacking form (column stacking form) of $A$.

\item $[x]$: Integer part of $x$.

\item $[a,b]$: the set of integers $a\leq i \leq b$.
\item $\d_n^i$: the $i$-th column of the identity matrix $I_n$.
%
%\item $\D_n:=\left\{\d_n^i\;|\; i=1,\cdots,n\right\}$.
%
\item $\d_n[i_1,\cdots,i_s]:=\left[\d_n^{i_1},\cdots,\d_n^{i_s}\right]$.
\item $\J_{\ell}:=(\underbrace{1,1,\cdots,1}_{\ell})^\mathrm{T}$.
\item ${\bf S}_d$: $d$-th order permutation group.
\item $\otimes$:  Kronecker product of matrices.
\item $\ltimes$: semi-tensor product (STP) of matrices.
\end{enumerate}

\section{Preliminaries}

\subsection{Semi-Tensor Product of Matrices}

This subsection is a brief review on the STP of matrices \cite{che12}.

\begin{dfn}\label{d2.1.1} Let $A\in {\cal M}_{m\times n}$, $B\in {\cal M}_{p\times q}$, and $t=\lcm(n,p)$. The STP of $A$ and $B$ is defined as
\begin{align}\label{2.1.1}
		A\ltimes B=\left(A\otimes I_{t/n}\right) \left(B\otimes I_{t/p}\right).
	\end{align}
\end{dfn}

The STP is a generalization of the classical matrix product, i.e., when $n=p$,  $A\ltimes B=AB$. Because of this, we  omit the symbol $\ltimes$ in most cases.

{\bf Throughout this paper we assume the default matrix product is  STP.}

One of the important advantages of STP is that it keeps most of the properties of the classical matrix product available. In the following we review some basic properties of the STP, which will be used in the sequel.

\begin{prp}\label{p2.1.2}
	\begin{itemize}
		\item[(i)] (Associativity) Let $A,B,C$ be three matrices of arbitrary dimensions. Then $(A\ltimes B)\ltimes C=A\ltimes (B\ltimes C)$.
		\item[(ii)] (Distributivity) Let $A,B$ be two matrices of same dimension, $C$ is of arbitrary dimension. Then $(A + B)\ltimes C=A\ltimes C + B\ltimes C$, $C\ltimes (A + B)=C\ltimes A + C\ltimes B$.
		\item[(iii)] Let $A,B$ be two matrices of arbitrary dimensions. Then $(A \ltimes B)^T=B^T\ltimes A^T$. If $A$, $B$ are invertible, then $(A \ltimes B)^{-1}=B^{-1}\ltimes A^{-1}$.
	\end{itemize}
\end{prp}

\begin{prp}\label{p2.1.3}
	Let $x\in \R^n$ be a column vector. Then $xA=(I_n\otimes A)x$.
		
Let $\omega\in \R^n$ be a row vector. Then $A\omega=\omega(I_n\otimes A)$.
\end{prp}

\begin{dfn}\label{d2.1.4} The $m\times n$ swap matrix is defined as follows:
\begin{align}\label{2.1.2}
W_{[m,n]}=\left[I_n\d_m^1,I_n\d_m^2,\cdots,I_n\d_m^m\right]\in {\cal M}_{mn\times mn}.
\end{align}
\end{dfn}

\begin{prp}\label{p2.1.5}
\begin{itemize}
\item[(i)]
\begin{align}\label{2.1.3}
W^{-1}_{[m,n]}=W^{\mathrm{T}}_{[m,n]}=W_{[n,m]}.
\end{align}
\item[(ii)]
Let $x\in \R^m$ and $y\in \R^n$ be two column vectors. Then
 \begin{align}\label{2.1.4}
W_{[m,n]}xy=yx.
\end{align}
\item[(iii)]
Let $\xi\in \R^m$ and $\eta\R^n$ be two row vectors. Then
 \begin{align}\label{2.1.5}
\xi\eta W_{[n,m]}=\eta \xi.
\end{align}
\end{itemize}
\end{prp}

\subsection{Hypermatrix}

\begin{dfn}\label{d2.2.1} \cite{lim13} For $n_1,\cdots,n_d\in \Z_+$, a function $f:[1,n_1]\times \cdots\times [1,n_d]\ra \R$ is a real hypermatrix of order $d$ and dimension $n_1\times \cdots\times n_d$.
\end{dfn}

Equivalently, a hypermatrix, ${\cal A}$, is  commonly considered as a set of ordered data with $i_s=[1,n_s]$, $s\in [1,d]$ as indexes, that is,
\begin{align}\label{2.2.1}
{\cal A}=\left\{a_{i_1,\quad,i_d}\;|\l i_s=[1,n_s],~s\in [1,d]\right\}\in \R^{n_1\times \cdots\times n_d}.
\end{align}
where $\R^{n_1\times \cdots\times n_d}$ is the set of hypermatrices of order $d$ and dimension $n_1\times\cdots\times n_d$.

Since the set of indexes play an important role in our following study, we first give a detailed discussion on the set of indexes.

An index, denoted by $\vec{i}$, is considered as a finite ordered set. The number of the set elements is called the length of the index, denoted by $|\vec{i}|$.
Two indexes are said to be equivalent, if there is an order-keeping one-to-one correspondence.
For example, if $\vec{i}=[1,n]$, $\vec{j}=[2,n+1]$, then $\vec{i}\sim\vec{j}$.

\begin{dfn}\label{d2.2.2} Let $\vec{i}=[1,m]$ and $\vec{j}=[1,n]$. The product of $\vec{i}$ and $\vec{j}$ is a double index, defined by
\begin{align}\label{2.2.2}
\vec{i}*\vec{j}:=\{(1,1),(1,2),\cdots,(1,n),(2,1),\cdots, (m,n)\},
\end{align}
where the elements are arranged in an alphabetic order.
\end{dfn}

Assume $\vec{k}=[1,mn]=\vec{i}\vec{j}$, where $*$ is omitted. Then it is easy to verify that
\begin{align}\label{2.2.3}
k=(i-1)n+j.
\end{align}
Conversely,
\begin{align}\label{2.2.4}
i=\left[(k-1)/n\right]+1,\quad j=k-in,
\end{align}
where $[x]$ is the integral part of $x$.

In general, we have the following index converting formula.

\begin{prp}\label{p2.2.3} Consider $\vec{i}_s=[1,n_s]$, $s\in [1,d]$. Assume $\vec{k}=\vec{i}_1\cdots\vec{i}_d$.
\begin{itemize}
\item[(i)]
\begin{align}\label{2.2.5}
k=(i_1-1)n_2\cdots n_d+(i_2-1)n_3\cdots n_d+\cdots+(i_{d-1}-1)n_d+i_d.
\end{align}
\item[(ii)]
\begin{align}\label{2.2.6}
\begin{cases}
k_0:=k-1,\\
k_{d-s+1}=\left[\frac{k_{d-s}}{n_{s}}\right],\\
i_s=k_{d-s}-k_{d-s+1}*n_s+1,\\~~\quad s=d,d-1,\cdots,2,\\
i_1=k_{d-1}+1.\\
\end{cases}
\end{align}
\end{itemize}
\end{prp}

If $|\vec{i}|=1$, then we denote it by $\vec{1}$, which  is a dummy index. It is obvious that for any index $\vec{k}$ we have
\begin{align}\label{2.2.7}
\vec{k}\vec{1}=\vec{1}\vec{k}=\vec{k}.
\end{align}

A straightforward verification shows that the product $*$ is associative, i.e., for any three indexes $\vec{i}$, $\vec{j}$, and $\vec{k}$
\begin{align}\label{2.2.8}
(\vec{i}\vec{j})\vec{k}=\vec{i}(\vec{j}\vec{k}).
\end{align}

Denote by $Id$ the set of indexes. (\ref{2.2.7}) and (\ref{2.2.8}) yield the following result.

\begin{prp}\label{p2.2.4}
$(Id,*)$ is a monoid (semi-group with identity).
\end{prp}

\begin{rem}\label{r2.2.401} Sometimes the dummy index $\vec{1}$ is useful. Say, assume $\vec{i}=\vec{i}_1\cdots\vec{i}_{2k-1}$ has odd factor indexes. If we want to split it into two equal factor subsets, a dummy index can be added as $\vec{i}_{2k}=\vec{1}$.
\end{rem}

To apply the methods and results developed in matrix theory to hypermatrices, the following  technique is fundamental to convert a hypermatrix into a matrix form, called the matricizing \cite{mar04}.

\begin{dfn}\label{d2.2.5} Consider a hypermatrix
$$
{\cal A}=\{a_{i_1,\cdots,i_d}\;|\; i_s\in[1,n_s],~s\in [1,d]\}\in \R^{n_1\times \cdots \times n_d}.
$$
 Assume
$\vec{i}=\vec{i}_1 \cdots \vec{i}_d$,  $\vec{j}=\vec{j}_1 \cdots \vec{j}_r$, $\vec{k}=\vec{k}_1 \cdots \vec{k}_s$, and
$\vec{i}=\vec{j} \vec{k}$
is a partition, where $r+s=d$. Then ${\cal A}$ can be arranged into a matrix, called the matricizing of ${\cal A}$, as
\begin{align}\label{2.2.9}%{3.1}
A=M^{\vec{j}\times \vec{k}} ({\cal A}),
\end{align}
where  $A=(a_{j,k})$ of dimension $n^j\times n^k$ with $n^j=|\vec{j}|$ and $n^k=|\vec{k}|$.
\end{dfn}

It is obvious that the matricizing of a hypermatrix is not unique. It depends on the index partition. We consider some special cases.

\begin{itemize}
\item[(i)] Vector Form:

If ${\bf k}=\vec{1}$ (or ${\bf j}=\vec{1}$), then ${\cal A}$ is arranged into a column (correspondingly, a row) vector as
\begin{align*}
V({\cal A})=M^{\vec{i}\times \vec{1}}({\cal A})=(a_{1,\cdots,1}, a_{1,\cdots,2},\cdots, a_{n_1,\cdots,n_d})^{\mathrm{T}}.
\end{align*}
Correspondingly,
\begin{align*}
V^{\mathrm{T}}({\cal A})=M^{\vec{1}\times \vec{i}}({\cal A})=(a_{1,\cdots,1}, a_{1,\cdots,2},\cdots, a_{n_1,\cdots,n_d}).
\end{align*}

\item[(ii)] Single-Index Form:

If $\vec{j}=\vec{i}_1$, it is a column single-index form, denoted by
\begin{align*}
	\begin{array}{l}
	M_r({\cal A})=M^{\vec{i}_1\times \vec{i}_2 \cdots \vec{i}_d}({\cal A})
	=\begin{bmatrix}
			a_{1,\cdots,1}&a_{1,\cdots,2}&\cdots&a_{1,n_2,\cdots,n_d}\\
			a_{2,\cdots,1}&a_{2,\cdots,2}&\cdots&a_{2,n_2,\cdots,n_d}\\
			\vdots&~&~&~\\
			a_{n_1,\cdots,1}&a_{n_1,\cdots,2}&\cdots&a_{n_1,n_2,\cdots,n_d}\\
		\end{bmatrix}
	\end{array}
\end{align*}
If $\vec{k}=\vec{i}_d$, it is a row single-index form, denoted by
\begin{align*}
	\begin{array}{l}
	M_c({\cal A})=M^{\vec{i}_1\cdots \vec{i}_{d-1} \times \vec{i}_d}({\cal A})
	=\begin{bmatrix}
			a_{1,\cdots,1,1}&a_{1,\cdots,1,2}&\cdots&a_{1,\cdots,1,n_d}\\
			a_{1,\cdots,2,1}&a_{1,\cdots,2,2}&\cdots&a_{1,\cdots,2,n_d}\\
			\vdots&~&~&~\\
			a_{n_1,\cdots,n_{d-1},1}&a_{n_1,\cdots,n_{d-1},2}&\cdots&a_{n_1,\cdots,n_{d-1},n_d}\\
		\end{bmatrix}
	\end{array}
\end{align*}

\end{itemize}

We give an example to describe this.

\begin{exa}\label{e2.2.6}

Given ${\cal A}=(a_{i_1,i_2,i_3})\in \R^{2\times 3\times 2}$. Consider $M^{\vec{j}\times \vec{k}}({\cal A})$.

\begin{itemize}
\item[(i)] $\vec{j}=\vec{i}_1\vec{i}_2\vec{i}_3$ and $\vec{k}=\vec{1}$:
$$
\begin{array}{l}
			V({\cal A})=M^{ \{i_1,i_2,i_3\}\times \vec{1}}({\cal A})=[a_{111},a_{112},a_{121},a_{122},\\
			~~~~~~~~a_{131},a_{132},a_{211},a_{212},a_{221},a_{222},a_{231},a_{232}]^{\mathrm{T}}.\\
\end{array}
$$

\item[(ii)] $\vec{j}=\vec{1}$ and $\vec{k}=\vec{i}_1\vec{i}_2\vec{i}_3$:
$$
V^{\mathrm{T}}({\cal A})=\left(V({\cal A})\right)^{\mathrm{T}}.
$$

\item[(iii)] $\vec{j}=\vec{i}_1$ and $\vec{k}=\vec{i}_2\vec{i}_3$:
$$
M_r({\cal A})=M^{\{i_1\}\times \{i_2,i_3\}}({\cal A})=
	\begin{bmatrix}
				a_{111}&a_{112}&a_{121}&a_{122}&a_{131}&a_{132}\\
				a_{211}&a_{212}&a_{221}&a_{222}&a_{231}&a_{232}
			\end{bmatrix}.\\
$$

\item[(iv)]  $\vec{j}=\vec{i}_1\vec{j}$ and $\vec{k}=\vec{i}_3$:
$$
M_c({\cal A})=	M^{\{i_1,i_2\}\times \{i_3\}}({\cal A})=
\begin{bmatrix}
			a_{111}&a_{112}\\
			a_{121}&a_{122}\\
			a_{131}&a_{132}\\
			a_{211}&a_{211}\\
            a_{221}&a_{222}\\
            a_{231}&a_{232}
		\end{bmatrix}.
$$
\end{itemize}
\end{exa}

\begin{rem}\label{r2.2.7}
\begin{itemize}
\item[(i)] A hypermatrix ${\cal A}$ has two kinds of expressions: One is an ordered set of date, such as in (\ref{2.2.1}), we call it the matrix set. The other one is a special matrix form under a fixed matricizing, which will be called the matrix form. Just like in matrix theory, the properties of matrices (including vectors) must be investigated with a special matrix form but the matrix set, most  properties of hypermatrices must be investigated with respect to some special matrix expressions.
\item[(ii)] Based on the perspective of (i), the order of a hypermatrix is of less importance. Particularly, when the dummy index is introduced, the order of a hypermatrix can easily be modified.
\end{itemize}
\end{rem}

\subsection{Hypervectors}

\begin{dfn}\label{d2.3.1} \cite{che25} Assume $x\in \R^n$ where $n=\prod_{k=1}^dn_k$. ($n_k>1$, $k\in [1,d]$)
$x$ is said to be a hypervector with respect to $n_1\times\cdots\times n_d$, if there exist $x_i\in \R^{n_i}$ such that
\begin{align}\label{2.3.1}
x=\ltimes_{i=1}^dx_i.
\end{align}
$x_i$, $i\in [1,d]$ are called the components of $x$. A hypervector is also called an order-1 form.

The set of such hypervectors is denoted by $\R^{n_1\ltimes \cdots\ltimes n_d}$, where $d$ is the order $n_1\times \cdots\times n_d$ is the dimension of the hypervector.
\end{dfn}

The basic properties of hypermatrices are listed as follows.

\begin{prp}\label{p2.3.2}  Assume $0\neq x\in \R^n$ where $n=\prod_{k=1}^dn_k$ ($n_k>1$, $k\in [1,d]$), and
$$
x=\ltimes_{k=1}^dx_k=\ltimes_{k=1}^dy_k,
$$
where $x_k,y_k\in\R^{n_k}$, $k\in [1,d]$.
Then
$$
y_k=c_kx_k,\quad k\in [1,d],
$$
and
$$
\prod_{k=1}^dc_k=1.
$$
\end{prp}

Using Proposition \ref{p2.3.2}, the following result is obvious.

\begin{cor}\label{c2.3.3} Let the hypervector $x=\ltimes_{k=1}^sx_k$, where $x_k\in \R^n$. Then $x$  uniquely determines a subspace
$$
V_x:=\Span\{x_1,\cdots,x_s\}\subset \R^n.
$$
\end{cor}

\begin{dfn}\label{d2.3.3} Let $0\neq x\in \R^n$. The head index of $x$ is defined by
\begin{align}\label{2.3.2}
e(x)=\min\{i\;|\;x(i)\neq 0\}.
\end{align}
and $x_{e(x)}$ is called the head value.
\end{dfn}

\begin{prp}\label{p2.3.4} Let $0\neq x\in \R^n$, $n=\prod_{k=1}^dn_k$, and assume
$$
x=\ltimes_{k=1}^dx_k,\quad x_k\in \R^{n_k}.
$$
Set $k=e(x)$. Then
$$
e(x_s)=i_s,\quad s\in [1,d],
$$
where  $i_s$ are provided by formula (\ref{2.2.6}).
\end{prp}

\section{KPD of Vector Form Hypermatrices}

\subsection{Exact KPD of Vector Form Hypermatrices}

\begin{dfn}\label{d3.1.1} Given a hypermatrix ${\cal A}=(a_{i_1,\cdots.i_d}\in \R^{n_1\times \cdots\times n_d}$.  The exact KPD means to find $x_s\in \R^{n_s}$, $s\in [1,d]$ such that
\begin{align}\label{3.1.1}
V({\cal A})=\ltimes_{s=1}^dx_s.
\end{align}
\end{dfn}

Assume ${\cal A}$ is given, then $V:=V({\cal A})\in \R^n$ is known, where $n=\prod_{s=1}^dn_s$.
If $V$ is decomposed to $V=\ltimes_{s=1}^dx_i$, using Proposition \ref{p2.3.4}, all the head indexes $e_s:=e(x_s)$ are known.

Define a set of projectors as
 \begin{align}\label{3.1.2}
\Xi^e_{s,n}:=\ltimes_{r=1}^{s-1}[\d_{n_r}^{e_r}]^{\mathrm{T}}\otimes I_{n_s}\otimes \ltimes_{r=s+1}^{d}[\d_{n_r}^{e_r}]^{\mathrm{T}},\quad s\in [1,d].
\end{align}

Then we have the following result \cite{che25,chepr}, which is called the monic decomposition algorithm (MDA).

\begin{thm}\label{p3.1.2} Let $V=V({\cal A})\in \R^n$. ${\cal A}$ with the head index $e$ and head value $h(V)$ and set
$V_0=V/h(V)$. Then ${\cal A}$ is (vector form) exactly decomposable, if and only if,
 \begin{align}\label{3.1.3}
V=h(V)\ltimes_{s=1}^d\Xi^e_{s,n}V_0.
\end{align}
\end{thm}

We give a simple example to describe this.

\begin{exa}\label{e3.1.3} Let ${\cal A}=(a_{i_1,i_2,i_3,i_4})\in \R^{4\times 2\times 2\times 3}$, where
$$
\begin{array}{llll}
a_{3,1,2,2}=4,&a_{3,1,2,3}=2,&a_{3,2,2,2}=8,&a_{3,2,2,3}=4,\\
a_{4,1,2,2}=-4,&a_{4,1,2,3}=-2,&a_{4,2,2,2}=-8,&a_{4,2,2,3}=-4,\\
\end{array}
$$
and all others are $0$. Then we have
$$
V=[\underbrace{0,\cdots,0}_{28},4,2, \underbrace{0,\cdots,0}_{4},8,4, \underbrace{0,\cdots,0}_{4},-4,-2, \underbrace{0,\cdots,0}_{4},-8,-4]^{\mathrm{T}}.
$$
$$
e=29;\quad h(V)=4.
$$
Using formula (\ref{2.2.6}), we have
$$
\begin{array}{l}
k_0=e-1=28,\\
k_1=\left[\frac{k_0}{n_4}\right]=\left[\frac{28}{3}\right]=9,\\
e_4=k_0-k_1n_4+1=2,\\
k_2=\left[\frac{k_1}{n_3}\right]=\left[\frac{9}{2}\right]=4,\\
e_3=k_2-k_2n_3+1=2,\\
k_3=\left[\frac{k_2}{n_2}\right]=\left[\frac{4}{2}\right]=2,\\
e_2=k_2-k_3n_2+1=1,\\
e_1=k_3+1=3.
\end{array}
$$
Then we construct the projectors as
\begin{align}\label{3.1.4}
\begin{array}{l}
\Xi^e_{1,n}:=I_{n_1}\otimes [\d_{n_2}^{e_2}\otimes \d_{n_3}^{e_3}\otimes \d_{n_4}^{e_4}]^{\mathrm{T}}
=I_{4}\otimes [\d_{2}^{1}\otimes \d_{2}^{2}\otimes \d_{3}^{2}]^{\mathrm{T}},\\
\Xi^e_{2,n}:=[\d_{n_1}^{e_1}]^{\mathrm{T}}\otimes I_{n_2}\otimes [\d_{n_3}^{e_3}\otimes \d_{n_4}^{e_4}]^{\mathrm{T}}
=[\d_{4}^3]^{\mathrm{T}} \otimes I_{2}\otimes [\d_{2}^{2}\otimes \d_{3}^{2}]^{\mathrm{T}},\\
\Xi^e_{3,n}:=[\d_{n_1}^{e_1}\otimes \d_{n_2}^{e_2}]^{\mathrm{T}}\otimes I_{n_3}\otimes [\d_{n_4}^{e_4}]^{\mathrm{T}}
=[\d_{4}^3\otimes\d_{2}^1]^{\mathrm{T}} \otimes I_{2}\otimes [\d_{3}^{2}]^{\mathrm{T}},\\
\Xi^e_{4,n}:=[\d_{n_1}^{e_1} \otimes \d_{n_2}^{e_2}\otimes \d_{n_3}^{e_3}]^{\mathrm{T}}\otimes I_{n_4}
=[\d_{4}^3\otimes \d_2^1\otimes\d_2^2]^{\mathrm{T}} \otimes I_{3}.\\
\end{array}
\end{align}
It follows that
$$
\begin{array}{l}
x_1=\Xi^e_{1,n}V_0=\Xi^e_{1,n}(V/h(V))=(0,0,1,-1)^{\mathrm{T}},\\
x_2=\Xi^e_{2,n}V_0=(1,2)^{\mathrm{T}},\\
x_3=\Xi^e_{3,n}V_0=(0,1)^{\mathrm{T}},\\
x_4=\Xi^e_{4,n}V_0=(0,1,0.5)^{\mathrm{T}}.\\
\end{array}
$$
Finally, it is ready to verify that
$$
V=h(V)\ltimes_{s=1}^4x_s.
$$
Hence, ${\cal A}$ is exact decomposable.

\end{exa}

\subsection{NKP of Vector Form Hypermatrices}

\begin{dfn}\label{d3.2.1} Given a hypermatrix ${\cal A}=(a_{i_1,\cdots.i_d})\in \R^{n_1\times \cdots\times n_d}$.  The NKP means to find $x^*_s\in \R^{n_s}$, $s\in [1,d]$ such that
\begin{align}\label{3.2.1}
(x^*_1,\cdots,x^*_d)=\argmin_{x_1,\cdots,x_d}\|V({\cal A})-\ltimes_{s=1}^dx_s\|_F.
\end{align}
\end{dfn}

It is easy to verify that  $\R^{n_1\ltimes \cdots\ltimes n_d}\subset\R^{n_1\times \cdots\times n_d}$ is a closed subset, the solution of NKP exists. According to Proposition \ref{p2.3.2}, up to a set of constant product parameters, the solution is unique.

\begin{lem}\label{l3.2.2} Assume $x^*_s\in \R^{n_s}$, $s\in [1,d]$ is the least square solution of (\ref{3.1.1}),  for each $s$,   we can fix  $x_k=x^*_k$, $k\neq s$, then
\begin{align}\label{3.2.2}
x^*_s=\argmin_{x_s}\|V({\cal A})-x^*_1\ltimes \cdots\ltimes x^*_{s-1}\ltimes x_s\ltimes x^*_{s+1}\ltimes \cdots\ltimes x^*_d\|_F.
\end{align}
\end{lem}

\noindent{\it Proof.} Assume $x^*_s$ is not the least square solution for (\ref{3.2.2}) and say, $\tilde{x}_s$ is the least square solution of (\ref{3.2.2}). Then it is easy to verify that $(x^*_1,\cdots,x^*_{s-1},\tilde{x}_s,x^*_{s+1},\cdots,x^*_d)$ is better than $(x^*_1,\cdots,x^*_d)$, which is a contradiction.

\hfill $\Box$

Consider the problem of minimizing the following Frobenius norm.
\begin{align}\label{3.2.3}
\begin{array}{l}
\min\left(\|V-\ltimes_{s=1}^dx_s\|_F\right)=\min\left(\|V-\ltimes_{s=1}^{d-1}x_s\ltimes x_d\|_F\right)\\
=\min \left(\|V-(\ltimes_{s=1}^{d-1} x_s \otimes I_{n_d}) x_d\|_F\right).
\end{array}
\end{align}

According to Lemma \ref{l3.2.2},  fix $x_s$, $s\in [1,d-1]$, then it becomes the problem of minimizing a linear form of $x_d$, then it is well known that the least square solution satisfies
\begin{align}\label{3.2.4}
\begin{array}{l}
(\ltimes_{s=1}^{d-1}x_s\otimes I_{n_d})^{\mathrm{T}}V\\
=(\ltimes_{s=1}^{d-1}x_s\otimes I_{n_d})^{\mathrm{T}}(\ltimes_{s=1}^{d-1}x_s\otimes I_{n_d})x_d\\
=\prod_{s=1}^{d-1}\|x_s\|^2x_d.\\
\end{array}
\end{align}
That is,
\begin{align}\label{3.2.5}
\begin{array}{l}
x_d=\frac{1}{\prod_{s=1}^{d-1}\|x_s\|^2}(\ltimes_{s=1}^{d-1}x_s\otimes I_{n_d})^{\mathrm{T}}V\\
~~=\frac{1}{\prod_{s=1}^{d-1}\|x_s\|^2}[(\ltimes_{s=1}^{d-1}x_s)^{\mathrm{T}})\otimes I_{n_d})]V\\
~~=\frac{1}{\prod_{s=1}^{d-1}\|x_s\|^2}[(\ltimes_{s=1}^{d-1}x_s)^{\mathrm{T}})\otimes I_{n_d})]V\\
~~=\frac{1}{\prod_{s=1}^{d-1}\|x_s\|^2} (\ltimes_{s=1}^{d-1}x_s)^{\mathrm{T}}) V\\
\end{array}
\end{align}

Next, we consider $x_s$ and set $u=\prod_{i=1}^{s-1}n_i$ and $v=\prod_{i=s+1}^dn_i$. Using Proposition \ref{p2.1.3} and swap matrix, we have
$$
\begin{array}{l}
\ltimes_{i=1}^dx_i=\ltimes_{i=1}^{s-1}x_i x_s \ltimes_{i=s+1}^dx_i\\
=\ltimes_{i=1}^{s-1}x_i W_{[v,n_s]}\ltimes_{i=s+1}^dx_i x_s\\
=(I_u\otimes W_{[v,n_s]}) \ltimes_{i=1,i\neq s}^dx_i x_s\\
\end{array}
$$
Putting this expression into (\ref{3.2.3}), a similar argument as in (\ref{3.2.4}) leads to
\begin{align}\label{3.2.6}
x_s=\frac{1}{\prod_{i=1,i\neq s}^{d}\|x_i\|^2}\left[ (\ltimes_{i=1}^{s-1}x_i)^{\mathrm{T}}\otimes I_{n_s}\otimes
 (\ltimes_{i=s+1}^{d}x_i)^{\mathrm{T}}\right]V,\quad s\in [1,d].
 \end{align}
Note that it is easy to verify that (\ref{3.2.5}) is a special case of (\ref{3.2.6}), so we do not need (\ref{3.2.5}) any more.

Summarizing the above argument, we conclude the following result.

\begin{prp}\label{p3.2.3} $(x_1,\cdots,x_s)$ is local least square solution of NKP, i.e., satisfies (\ref{3.2.1}) locally, if and only if,
it satisfies (\ref{3.2.6}).
\end{prp}

Next, we give a numerical algorithm, called the stationary value based algorithm (SVA), to solve (\ref{3.2.6}).

\begin{alg}\label{a3.2.4}
\begin{itemize}
\item[] Step 1.

Choose initial value randomly as (say, in MatLab)
$$
x^0_s=rand(n_s,1),\quad s\in [1,d].
$$
or
$$
x^0_s=rand(n_s,1)-0.5*ones(n_s,1),\quad s\in [1,d].
$$
\item[] Step k. Assume $x^{k-1}_s$, $s\in [1,d]$ are known.

\begin{itemize}
\item[(i)] Fix $x^{k-1}_s$, $s\in [2,d]$, and use  (\ref{3.2.6}) to solve $x^k_1$ out.
\item[(ii)] Fix $x^k_1$ and $x^{k-1}_s$, $s\in [3,d]$, and use (\ref{3.2.6}) to solve $x^k_2$ out.
\item[(iii)] In general, fix updated $x^k_i$, $i<t$ and $x^{k-1}_i$, $i>t$,  and use (\ref{3.2.6}) to solve $x^k_t$, $t\in [1,d]$.
\end{itemize}

\item[] Last Step. For a pre-assigned $0<\epsilon << 1$, if
$$
\|\ltimes_{s=1}^dx^{k}_s-\ltimes_{s=1}^dx^{k-1}_s\|_F<\epsilon,
$$
stop. Otherwise, go back to Step 1.
\end{itemize}
\end{alg}

\begin{rem}\label{r3.2.5}
\begin{itemize}
\item[(i)] Using Algorithm \ref{a3.2.4}, the square error $\|V-h(V)\ltimes_{s=1}^dx^k_s\|$ is monotonically decreasing, as $k\ra\infty$, it will converge to a stationary $(x^*_1,\cdots,x^*_d)$, which is a local minimum solution.
\item[(ii)] It is obvious that the set of  stationary points is uncountable. But numerical experiments suggest that the stationary values are few. In most cases there is only one stationary value. In this case, Algorithm  \ref{a3.2.4} provides the least square solution.
\end{itemize}
\end{rem}

We give a numerical example to depict this.

\begin{exa}\label{e3.2.5}  Let ${\cal A}=(a_{i_1,i_2,i_3,i_4})\in \R^{4\times 2\times 2\times 3}$ as in Example \ref{e3.1.3}, but with
$$
\begin{array}{llll}
a_{3,1,2,2}=-2,&a_{3,1,2,3}=3.5,&a_{3,2,2,2}=-5.2,&a_{3,2,2,3}=7.3,\\
a_{4,1,2,2}=0.5,&a_{4,1,2,3}=2,&a_{4,2,2,2}=6.5,&a_{4,2,2,3}=-5,\\
\end{array}
$$
and all others are $0$.

Set $V=V({\cal A})$, then $e(V)=29$ and $h(V)=-2$.
Let $V_0:=V/h(V)=-V/2$. Since $e(V)=29$ is the same as in Example \ref{e3.1.3}, we can use the projectors $\Xi^e_{s,n}$, $s\in[1,4]$, defined in (\ref{3.1.4}) to get
\begin{align}\label{3.2.7}
\begin{array}{l}
x_1=\Xi^e_{1,n}V_0=[0,0,1,-0.25]^{\mathrm{T}},\\
x_2=\Xi^e_{2,n}V_0=[1,2.6]^{\mathrm{T}},\\
x_3=\Xi^e_{3,n}V_0=[0,1]^{\mathrm{T}},\\
x_4=\Xi^e_{4,n}V_0=[0,1,-1.75]^{\mathrm{T}}.\\
\end{array}
\end{align}

Then we have
$$
\|V=h(V)\ltimes_{k=1}^4x_k\|_F=6.7802.
$$
Hence this ${\cal A}$ is not decomposable.

Using Algorithm \ref{a3.2.4}.  Using any initial values $x^0_1\in \R^4$,  $x^0_2\in \R^2$,$x^0_3\in \R^3$, and $x^0_4\in \R^3$, and choosing $\epsilon=10^{-4}$, each time we have different solutions, but the least square root
$$
Error=\|V-h(V)\ltimes_{s=1}^4x*_s\|_F=4.3218.
$$
We conclude that the least square error is $4.3218$. One such decomposition is
$$
\begin{cases}
x_1=(0,0,-245.3272,200.2378)^{\mathrm{T}},\\
x_2=(0.1041,0.5358)^{\mathrm{T}},\\
x_3=(0, -0.1451)^{\mathrm{T}},\\
x_4=(0,-0.3328,0.3561)^{\mathrm{T}},\\
\end{cases}
$$
\end{exa}

For ill-conditioned data, the Monte Carlo-like  simulation can provide the least square solution.

We give an example to demonstrate this.

\begin{exa}\label{e3.2.6}   Let ${\cal A}=(a_{i_1,i_2,i_3,i_4})\in \R^{4\times 2\times 2\times 3}$, where
$$
\begin{array}{llll}
a_{3,1,2,2}=2,&a_{3,2,1,1}=3.5,&a_{4,1,1,3}=-5.2,&a_{4,1,2,1}=7.3,\\
a_{4,2,1,2}=0.5.&a_{4,2,1,3}=2,&a_{4,2,2,1}=6.5,&a_{4,2,2,2}=-5,\\
\end{array}
$$
all other entries are zero.

Using Algorithm \ref{a3.2.4}, we calculate the stationary points with their Errors.

Choose $\epsilon=10^{-4}$ and
$$
\begin{array}{ll}
x^0_1=rand(4,1)-0.5*ones(4,1),&x^0_2=rand(2,1)-0.5*ones(2,1),\\
x^0_3=rand(2,1)-0.5*ones(2,1),&x^0_4=rand(3,1)-0.5*ones(3,1).
\end{array}
$$

The algorithm has been repeated 1000 times: Set $Error:=\|V-\ltimes_{s=1}^4x_s\|_F$,
the answer is:
\begin{itemize}
\item[(i)]
$$
\begin{array}{l}
 Error=7.7168,\\
 No=911,\\
\begin{cases}
x_1=(0,0,3.9440,-74.6696)^{\mathrm{T}},\\
x_2=(-0.4186,-0.4847)^{\mathrm{T}},\\
x_3=(0.0079,-0.2967)^{\mathrm{T}},\\
x_4=(-0.6808,-0.2713,-0.0035)^{\mathrm{T}},\\
\end{cases}
\end{array}
$$

\item[(ii)]

$$
\begin{array}{l}
 Error=11.7043,\\
 No=88,\\
\begin{cases}
x_1=(0,0,0.8540,73.7217)^{\mathrm{T}},\\
x_2=(0.3323,-0.1289)^{\mathrm{T}},\\
x_3=(0.4407,-0.0239)^{\mathrm{T}},\\
x_4=(-0.0221,-0.0241,-0.4798)^{\mathrm{T}},\\
\end{cases}
\end{array}
$$

\item[(iii)]
$$
\begin{array}{l}
 Error=11.7130,\\
 No=1,\\
\begin{cases}
x_1=(0,0,-2.7910,-48.6546)^{\mathrm{T}},\\
x_2=(0.6496,-0.1872)^{\mathrm{T}},\\
x_3=(0.2309,-0.0705)^{\mathrm{T}},\\
x_4=(0.1978,0.0714,0.6664)^{\mathrm{T}},\\
\end{cases}
\end{array}
$$
\end{itemize}

We have chosen
$$
\begin{array}{ll}
x^0_1=rand(4,1),&x^0_2=rand(2,1),\\
x^0_3=rand(2,1),&x^0_4=rand(3,1),
\end{array}
$$
or other initial values. All the simulations show the similar results, and no other Errors can be found.

We conclude that the smallest
$$
Error =7.7168.
$$
The $(x_1,x_2,x_3,x_4)$ in (i) is one least square solution.
\end{exa}

\subsection{Finite Sum KPD via NKP of Vector Form Hypermatrices}

Consider ${\cal A}\in \R^{n_1\times \cdots\times n_d}$. Repeatedly using NKP algorithm we can obtain a finite sum KPD. The algorithm is proposed as follows.

\begin{alg}\label{a3.3.1}
\begin{itemize}
\item[] Step1. Set $V_0=V({\cal A})$.
\item[] Step k. Assume $V_{k-1}$ is known.
\begin{itemize}
\item[] Step k.1.
Using NKP to find the NKP solution
$$
V_{k-1}\approx \ltimes_{s=1}^dx^k_s.
$$
\item[] Step k.2. Set
$$
V_k=V_{k-1}-\ltimes_{s=1}^dx^k_s.
$$

\item[] Step k.3. Check whether
$$
\|V_k\|_F<\epsilon?
$$
If ``Yes", Stop.
If ``No", Go to Step k.1.
\end{itemize}
\end{itemize}
\end{alg}

\begin{exa}\label{e3.3.2} Recall Example \ref{e3.2.5}. Using the same data and applying Algorithm \ref{a3.3.1} to it, we have the following result:
\begin{itemize}
\item[] Step 1:
$$
\begin{array}{l}
\begin{cases}
x^1_1=(0,0,-4.2261,3.4495)^{\mathrm{T}},\\
x^1_2=(0.0672,0.3460)^{\mathrm{T}},\\
x^1_3=(0,-3.3633)^{\mathrm{T}},\\
x^1_4=(0,-1.2904,1.3810)^{\mathrm{T}}.
\end{cases}\\
Error_1=4.3218.
\end{array}
$$
\item[] Step 2:
$$
\begin{array}{l}
\begin{cases}
x^2_1=(0,0,-19.0525,-25.9631)^{\mathrm{T}},\\
x^2_2=(-0.4380, -0.0489)^{\mathrm{T}},\\
x^2_3=(0,-0.6657)^{\mathrm{T}},\\
x^2_4=(0,0.0702, -0.4053)^{\mathrm{T}}.
\end{cases}\\
Error_2=1.8901.
\end{array}
$$
\item[] Step 3:
$$
\begin{array}{l}
\begin{cases}
x^3_1=(0,0,-12.4658,-13.6873)^{\mathrm{T}},\\
x^3_2=(0.0544,-0.4212)^{\mathrm{T}},\\
x^3_3=(0,0.3124)^{\mathrm{T}},\\
x^3_4=(0,0.7474,0.1320)^{\mathrm{T}}.
\end{cases}\\
Error_3=0.3104.
\end{array}
$$
\item[] Step 4:
$$
\begin{array}{l}
\begin{cases}
x^4_1=(0,0,-1.7675,1.6098)^{\mathrm{T}},\\
x^4_2=(-0.4504,-0.0589)^{\mathrm{T}},\\
x^4_3=(0,0.5602)^{\mathrm{T}},\\
x^4_4=(0,-0.4928,-0.0833)^{\mathrm{T}}.
\end{cases}\\
Error_3=0.0623
\end{array}
$$
\item[] Step 5:
$$
\begin{array}{l}
\begin{cases}
x^5_1=(0,0,-0.0162,0.0148)^{\mathrm{T}},\\
x^5_2=(-0.1803,1.3796)^{\mathrm{T}},\\
x^5_3=(0,0.6750)^{\mathrm{T}},\\
x^5_4=(0,0.5051,-2.9820)^{\mathrm{T}}.
\end{cases}\\
Error_5=1.1103*10^{-4}.
\end{array}
$$
\item[] Step 6:
$$
\begin{array}{l}
\begin{cases}
x^6_1=10^{-3}*(0,0,-0.2742,-0.3011)^{\mathrm{T}},\\
x^6_2=(0.5196,0.0679)^{\mathrm{T}},\\
x^6_3=(0,-0.6836)^{\mathrm{T}},\\
x^6_4=(0,-0.1271,0.7504)^{\mathrm{T}}.
\end{cases}\\
Error_6=1.1712*10^{-9}.
\end{array}
$$
\end{itemize}
We conclude that
$$
V({\cal A})=\dsum_{k=1}^6\ltimes_{s=1}^4x^k_s,
$$
with
$$
Error=1.1712*10^{-9},
$$
which might be caused by numerical computation error.
\end{exa}

\section{KPD for Matrix Form Hypermaxices}

\subsection{Matrix Form Decomposability}

Assume ${\cal A}=(a_{i_1,\cdots,i_d})\in \R^{n_1\times \cdots\times n_d}$ and $A=M^{\vec{j}\times \vec{k}}({\cal A})$, where
\begin{align}\label{4.1.1}
\vec{j}=\vec{i}_1\vec{i}_2\cdots\vec{i}_r,\quad \vec{k}=\vec{i}_{r+1}\vec{i}_{r+2}\cdots\vec{i}_d,\quad 1\leq r<d.
\end{align}
Set $p=\prod_{s=1}^rn_s$ and $q=\prod_{s=r+1}^dn_s$, then $A=(a^{j,k})\in {\cal M}_{p\times q}$.
By definition of $\vec{j}$ and $\vec{k}$, it is clear that the rows of $A$ are labeled by index  $\vec{i}_1\vec{i}_2\cdots\vec{i}_r$ and the columns of $A$ are labeled by index $\vec{i}_{r+1}\vec{i}_{r+2}\cdots\vec{i}_d$.

Consider $V=V_r(A)$. Then it is clear that the elements of $V$ is labelled by index of $\vec{i}_1\vec{i}_2\cdots\vec{i_d}$.

Now we assume $d=2r$ is even, and let the index partition be as in (\ref{4.1.1}), i.e., set
$$
\vec{j}_s=\vec{i}_s,\quad \vec{k}_s=\vec{i}_{r+s},\quad s\in [1,r].
$$

We consider the exact KPD problem of ${\cal A}$. That is, whether there exist $A_s\in {\cal M}_{n_s\times n_{r+s}}$, $s\in [1,r]$, such that
\begin{align}\label{4.1.2}
A=A_1\otimes A_2\otimes \cdots\otimes A_r?
\end{align}

We try to solve matrix form KPD by using the technique for vector form KPD.
To this end, we consider
$V_r(A)$ and $\ltimes_{s=1}^rV_r(A_s)$.
Note that (\ref{4.1.2}) does imply $V_r(A)=\ltimes_{s=1}^rV_r(A_s)$!

To convert the matrix form KPD to vector form KPT, assume (\ref{4.1.2}) holds, then the entries of $V_r(A)$ and
$\ltimes_{s=1}^r(V_r(A_s))$ are the same as
$$
\left\{a^1_{j_1,k_1}a^2_{j_2,k_2}\cdots a^r_{j_r,k_r}\;|\;j_s\in [1,n_s],~k_s\in[1,n_{r+s}],\quad s\in[1,r]\right\}.
$$
But in  $V_r(A)$, they are arranged in the order of $\vec{j}_1\cdots\vec{j}_r\vec{k}_1\cdots\vec{k}_r$, while in
$\ltimes_{s=1}^r(V_r(A_s))$ they are arranged in the order of $\vec{j}_1\vec{k}_1\vec{j}_2\vec{k}_2\cdots\vec{j}_r\vec{k}_r$.

To matching them we need a tool called the permutation matrix \cite{chepr}.

\begin{prp}\label{p4.1.1} Let $x_s\in \R^{n_s}$, $s\in [1,d]$, $n=\prod_{s=1}^dn_s$, and $\sigma\in {\bf S}_d$ be a permutation. Then there exists a unique matrix, $W^{\sigma}_{[n_1\times \cdots\times n_d]}\in {\cal M}_{n\times n}$, called the permutation matrix, such that
\begin{align}\label{4.1.3}
W^{\sigma}_{[n_1\times \cdots\times n_d]}\ltimes_{s=1}^dx_s=\ltimes_{s=1}^dx_{\sigma(s)}.
\end{align}
\end{prp}

The $W^{\sigma}_{[n_1\times \cdots\times n_d]}$ can be constructed as follows:
Set
$$
D_{\sigma}:=\left\{\d^{j_1}_{n_{\sigma(1)}}\d^{j_2}_{n_{\sigma(2)}}\cdots \d^{j_d}_{n_{\sigma(d)}}\;|\; j_k\in [1,n_{\sigma(k)}],\; k\in [1,d]\right\}.
$$

Then arrange the vectors in $D$ in the order as
$
\vec{j}_{\sigma^{-1}(1)}\vec{j}_{\sigma^{-1}(2)}\cdots\vec{j}_{\sigma^{-1}(d)}.
$
\footnote{The toolbox of STP can be found at $http:/lsc.amss.ac.cn/\sim hsqi/soft/STP.zip$. Then
$$
W^{\sigma}_{[n_1,\cdots,n_d]}=npermax([n_1,\cdots,n_d],[\sigma(1),\cdots,\sigma(d)]).
$$
}

Now since in $V_r(A)$ the entries are arranged by the order $\vec{j}_1\cdots\vec{j}_r\vec{k}_1\cdots\vec{k}_r$ and in
$\ltimes_{s=1}^rV_r(A_s)$ the entries are arranged by the order $\vec{j}_1\vec{k_1}\cdots\vec{j}_r\vec{k}_r$. We need a permutation
$$
\sigma:(1,2,\cdots,r,r+1,\cdots,2r)\mapsto (1,r_1,2,r+2,\cdots, 2r-1,2r).
$$
Hence we construct a permutation matrix as
\begin{align}\label{4.1.4}
W^{\sigma}_{[n_1,n_2,\cdots,n_{2r}]}.
\end{align}

Then the following result is obvious.

\begin{prp}\label{p4.1.2}
(\ref{4.1.2}) holds, if and only if,
\begin{align}\label{4.1.5}
W^{\sigma}_{[n_1,n_2,\cdots,n_{2r}]} V_r(A)=\ltimes_{s=1}^d V_r(A_s).
\end{align}
\end{prp}

As an immediate consequence, we have

\begin{cor}\label{c4.1.3} A hypermatrix ${\cal A}=(a_{i_1,\cdots,i_{2r}})\in \R^{n_1\times \cdots\times n_{2r}}$ is matrix form decomposable, as shown in (\ref{4.1.2}), if and only if,
$W^{\sigma}_{[n_1,n_2,\cdots,n_{2r}]} V_r(A)$ is decomposable, with respect to $(n_1n_{r+1}),(n_2n_{r+2}),\cdots,(n_rn_{2r})$.
\end{cor}

In fact, Corollary \ref{c4.1.3} claims that the KPD of matrix form hypermatrices can be converted to the KPD of vector form hypermatrices. Hence, the results developed in previous section are all applicable to KPD of matrix form hypermatrices.

Observing Proposition \ref{p2.3.2}, we can also have the following result about the uniqueness of KPD of matrix form hypermatrices.

\begin{cor}\label{c4.1.4} Up to a set of constant product coefficients the KPD of  matrix form hypermatrices is unique. That is, assume
$$
A=A_1\otimes \cdots \otimes A_d=B_1\otimes\cdots\otimes B_d,
$$
where $A_s$ and $B_s$, $s\in[1,d]$,  are of the same dimensions. Then there exist a set of parameters $c_s$, $s\in [1,d]$, such that
$B_s=c_sA_s$, $s\in [1,d]$ and $\prod_{s=1}^dc_s=1$.
 \end{cor}

\subsection{A Numerical Example}

The following matrix is proposed by \cite{col62}.
$$
A=\left[\begin{array}{cccccccccccccccc}
  1&17&33&49&65&81&97  &113&128&112&96&80&64&48&32  &16\\
  2&18&34&50&66&82&98  &114&127&111&95&79&63&47&31  &15\\
  3&19&35&51&67&83&99  &115&126&110&94&78&62&46&30  &14\\
  4&20&36&52&68&84&100&116&125&109&93&77&61&45&29 &13\\
  5&21&37&53&69&85&101&117&124&108&92&76&60&44&28 &12\\
  6&22&38&54&70&86&102&118&123&107&91&75&59&43&27 &11\\
  7&23&39&55&71&87&103&119&122&106&90&74&58&42&26 &10\\
  8&24&40&56&72&88&104&120&121&105&89&73&57&41&25 &9\\
  9&25&41&57&73&89&105&121&120&104&88&72&56&40&24&8\\
 10&26&42&58&74&90&106&122&119&103&87&71&55&39&23&7\\
11&27&43&59&75&91&107&123&118&102&86&70&54&38&22&6\\
12&28&44&60&76&92&108&124&117&101&85&69&53&37&21&5\\
13&29&45&61&77&93&109&125&116&100&84&68&52&36&20&4\\
14&30&46&62&78&94&110&126&115& 99&83& 67&51&35&19&3\\
15&31&47&63&79&95&111&127&114& 98&82& 66&50&34&18&2\\
16&32&48&64&80&96&112&128&113& 97&81& 65&49&33&17&1\\
\end{array}\right].
$$

\begin{itemize}
\item[(1)], $d=2$ decomposition:
\end{itemize}

Using Kronecker product singular value decomposition (KPSVD), a finite sum KPD is obtained as \cite{bat17}

\begin{align}\label{kd.2.4}
A=1154.2B_1\otimes C_1+117.98B_2\otimes C_2,
\end{align}
where
$$
B_1=
\begin{bmatrix}
0.09&0.31&0.35&0.13\\
0.10&0.33&0.34&0.12\\
0.12&0.34&0.33&0.10\\
0.13&0.35&0.31&0.09\\
\end{bmatrix}
$$
$$
C_1=
\begin{bmatrix}
0.25&0.25&0.25&0.25\\
0.25&0.25&0.25&0.25\\
0.25&0.25&0.25&0.25\\
0.25&0.25&0.25&0.25\\
\end{bmatrix}
$$
$$
B_2=
\begin{bmatrix}
0.25&0.25&-0.25&-0.25\\
0.25&0.25&-0.25&-0.25\\
0.25&0.25&-0.25&-0.25\\
0.25&0.25&-0.25&-0.25\\
\end{bmatrix}
$$
$$
C_2=
\begin{bmatrix}
-0.35&-0.13&0.09&0.31\\
-0.34&-0.12&0.10&0.33\\
-0.33&-0.10&0.12&0.34\\
-0.31&-0.09&0.31&0.35\\
\end{bmatrix}
$$

It is easy to calculate the error as
\begin{align}\label{kd.2.6}
Error=\|A-1154.2B_1\otimes C_1+117.98B_2\otimes C_2\|_F=170.45.
\end{align}

Using MDA, the result is\cite{chepr}
\begin{align}\label{kd.2.8}
\begin{array}{l}
B_1=\begin{bmatrix}
1&65&128&64\\
5&69&124&60\\
9&73&120&56\\
13&77&116&52\\
\end{bmatrix},\\
C_1=\begin{bmatrix}
1&17&33&49\\
2&18&34&50\\
3&19&35&51\\
4&20&36&52\\
\end{bmatrix},\\
B_2=\begin{bmatrix}
0&1&2.0156&1.0156\\
0.0625&1.0625&1.9531&0.9531\\
0.1250&1.1250&1.8906&0.8906\\
0.1875&1.1875&1.8281&0.8281\\
\end{bmatrix},\\
C_2=\begin{bmatrix}
0&1&2&3\\
0.0625&1.0625&2.0625&3.0625\\
0.1250&1.1250&2.1250&3.1250\\
0.1875&1.1875&2.1875&3.1875\\
\end{bmatrix}.\\
\end{array}
\end{align}

Finally, we can check that
\begin{align}\label{kd.2.10}
A=B_1\otimes C_1-1024B_2\otimes C2.
\end{align}
According to MatLab, the error
$$
Error<10^{-30}.
$$

\begin{itemize}
\item[(2)], $d=4$ decomposition:
\end{itemize}

Using KPSVD, a finite sum KPD is obtained as \cite{bat17}
\begin{align}\label{4.2.1}
\begin{array}{l}
A^*=1033.98\begin{bmatrix}0.47&0.53\\0.53&0.47\end{bmatrix}\otimes
\begin{bmatrix}0.5&0.5\\0.5&0.5\end{bmatrix}\otimes
\begin{bmatrix}0.5&0.5\\0.5&0.5\end{bmatrix}\otimes
\begin{bmatrix}0.5&0.5\\0.5&0.5\end{bmatrix}\\
~~+513.00\begin{bmatrix}0.5&-0.5\\0.5&-0.5\end{bmatrix}\otimes
\begin{bmatrix}-0.53&0.47\\-0.47&0.53\end{bmatrix}\otimes
\begin{bmatrix}0.5&0.5\\0.5&0.5\end{bmatrix}\otimes
\begin{bmatrix}0.5&0.5\\0.5&0.5\end{bmatrix}\\
~~+256.50\begin{bmatrix}0.5&-0.5\\0.5&-0.5\end{bmatrix}\otimes
\begin{bmatrix}0.5&0.5\\0.5&0.5\end{bmatrix}\otimes
\begin{bmatrix}-0.53&0.47\\-0.47&0.53\end{bmatrix}\otimes
\begin{bmatrix}0.5&0.5\\0.5&0.5\end{bmatrix}\\
~~+128.25\begin{bmatrix}0.5&-0.5\\0.5&-0.5\end{bmatrix}\otimes
\begin{bmatrix}0.5&0.5\\0.5&0.5\end{bmatrix}\otimes
\begin{bmatrix}0.5&0.5\\0.5&0.5\end{bmatrix}\otimes
\begin{bmatrix}-0.53&0.47\\-0.47&0.53\end{bmatrix}\\
\end{array}
\end{align}

It is easy to calculate the square error as
$$
ERROR:=V(A)-V(A^*))^{\mathrm{T}}(V(A)-V(A^*)=11.033.
$$

Using MDA, item number $k=4$ can not be obtained (although $k=8$ decomposition is obtained with $ERROR<10^{-10}$)\cite{chepr}.

Next, we use NKP to solve the same problem.

First, we have to construct $W^{\sigma}_{[2,2,2,2,2,2,2,2]}$, where
$$
\sigma: \vec{i}_1\vec{i}_2\vec{i}_3\vec{i}_4\vec{i}_5 \vec{i}_6\vec{i}_7\vec{i}_8\mapsto
\vec{i}_1\vec{i}_5\vec{i}_2\vec{i}_6\vec{i}_3 \vec{i}_7\vec{i}_4\vec{i}_8.
$$
Since $\vec{i}_1$ and $\vec{i}_8$ did not change, we need only to construct
$$
\begin{array}{l}
\sigma_0: \vec{i}_2\vec{i}_3\vec{i}_4\vec{i}_5 \vec{i}_6\vec{i}_7\mapsto
\vec{i}_5\vec{i}_2\vec{i}_6\vec{i}_3 \vec{i}_7\vec{i}_4\\
\Leftrightarrow (1,2,3,4,5,6)\mapsto (4,1,5,2,6,3).
\end{array}
$$
Then the permutation matrix can be constructed as
$$
\begin{array}{l}
W^{\sigma_0}_{[2,2,2,2,2,2]}=\d_{64}[1,3,9,11,33,35,41,43,2,4,10,12,34,36,42,44,5,7,13,\\
15,37,39,45,47,6,8,14,16,38,40,46,48,17,19,25,27,49,51,57,59,18,\\
20,26,28,50,52,58,60,21,23,29,31,53,55,61,63,22,24,30,32,54,56,62,64].
\end{array}
$$
Then
$$
W^{\sigma}_{[2,2,2,2,2,2,2,2]}=I_2\otimes W^{\sigma_0}_{[2,2,2,2,2,2]}\otimes I_2.
$$
Define
$$
V_0=W^{\sigma}_{[2,2,2,2,2,2,2,2]}V_r(A).
$$
\begin{itemize}
\item[(3)] Step 1:

Using Algorithm \ref{a3.3.1} to $V_0$ yields
$$
\begin{array}{l}
V_r(A^1_1)=1000*(1.1666,1.3209,1.3209,1.1666)^{\mathrm{T}},\\
V_r(A^1_2)=(0.5397,0.5397,0.5397,0.5397)^{\mathrm{T}},\\
V_r(A^1_3)=(0.2294, 0.2294,0.2294,0.2294)^{\mathrm{T}},\\
V_r(A^1_4)=(0.4190,0.4190,0.4190.0.4190)^{\mathrm{T}},\\
\end{array}
$$
Hence
$$
\begin{array}{ll}
A^1_1=1000\begin{bmatrix}1.1666&1.3209\\1.3209&1.1666\end{bmatrix},&
A^1_2=0.5397J_2,\\
A^1_3=0.2294J_2,&A^1_4=0.4190J_2,
\end{array}
$$
where
$$
J_2=\begin{bmatrix}
1&1\\1&1\end{bmatrix}.
$$
$$
ERROR=\|A-A^1_1\otimes A^1_2\otimes A^1_3\otimes A^1_4\|^2_F=345408.
$$
\item[] Step 2:
Setting $V_1=V_0-\ltimes_{s=1}^4 V_r(A^1_s)$ and using  Algorithm \ref{a3.3.1} to $V_1$ yield
$$
\begin{array}{ll}
A^2_1=-72.74J_2,&
A^2_2=\begin{bmatrix}4.8642&-4.2920\\4.2920&-4.8642\end{bmatrix},\\
A^2_3=0.2294J_2,&A^2_4=0.4190J_2.
\end{array}
$$
$$
ERROR=\|A_1-A^2_1\otimes A^2_2\otimes A^2_3\otimes A^2_4\|^2_F=82240.
$$
\item[] Step 3:
Setting $V_2=V_1-\ltimes_{s=1}^4 V_r(A^2_s)$ and using  Algorithm \ref{a3.3.1} to $V_2$ yield
$$
\begin{array}{ll}
A^3_1=10\begin{bmatrix}0.34&-0.34\\0.34&-0.34\end{bmatrix},&
A^3_2=9.9798J_2\\
A^3_3=\begin{bmatrix}-1.1957&1.0551\\-1.0551&1.1957\end{bmatrix},&
A^3_4=0.4190J_2.
\end{array}
$$
$$
ERROR=\|A_2-A^3_1\otimes A^3_2\otimes A^3_3\otimes A^3_4\|^2_F=16448.
$$
\item[] Step 4:
Setting $V_3=V_2-\ltimes_{s=1}^4 V_r(A^3_s)$ and using  Algorithm \ref{a3.3.1} to $V_3$ yield
$$
\begin{array}{ll}
A^4_1=\begin{bmatrix}-0.8027&0.8027\\-0.8027&0.8027\end{bmatrix},&
A^4_2=1.0110J_2\\
A^4_3=-5.9857J_2&A^4_4\begin{bmatrix}-1.7498&1.5439\\-1.5439&1.7498\end{bmatrix}.
\end{array}
$$
$$
ERROR=\|A_3-A^4_1\otimes A^4_2\otimes A^4_3\otimes A^4_4\|^2_F=1.5799*10^{-25}.
$$
\end{itemize}

\begin{itemize}
\item[(8)], $d=8$ decomposition:
\end{itemize}
Find $U^k_i\in {\cal M}_{2,1}$, $V^k_i\in {\cal M}_{1,2}$, $i\in [1,4]$, and $k$ is the smallest one such that
\begin{align}\label{4.2.2}
A=\dsum_{i=1}^kU^k_1\otimes V^k_1\otimes U^k_2\otimes V^k_2\otimes U^k_3\otimes V^k_3+U^k_4\otimes V^k_4.
\end{align}

We use indirect decomposition. We need an auxiliary result, which is straightforward verifiable.

\begin{prp}\label{p4.2.1}
 Let
$$
M=\begin{bmatrix}
a&b\\c&d\end{bmatrix}
$$
Then
\begin{itemize}
\item[(i)] $a\neq 0$:
\begin{align}\label{4.2.3}
M=\begin{bmatrix}1\\c/a\end{bmatrix}\begin{bmatrix}a&b\end{bmatrix}+
\begin{bmatrix}0\\1\end{bmatrix}\begin{bmatrix}0&d-cb/a\end{bmatrix}.
\end{align}
\item[(ii)] $a=0,~b\neq 0$:
\begin{align}\label{4.2.4}
M=\begin{bmatrix}1\\d/b\end{bmatrix}\begin{bmatrix}0&b\end{bmatrix}+
\begin{bmatrix}0\\1\end{bmatrix}\begin{bmatrix}c&0\end{bmatrix}.
\end{align}
\item[(iii)] $a=0,~b=0$:
\begin{align}\label{4.2.5}
M=\begin{bmatrix}0\\1\end{bmatrix}\begin{bmatrix}c&d\end{bmatrix}.
\end{align}
\end{itemize}
\end{prp}

Using Proposition \ref{p4.2.1} and the result obtained in $d=4$ decomposition, we have
\begin{itemize}
\item[(i)]
$$
A^1_1=1000\begin{bmatrix}1.1666&1.3209\\1.3209&1.1666\end{bmatrix}=U^1_1\otimes V^1_1+U^2_1\otimes V^2_1,
$$
where
$$
\begin{array}{l}
U^1_1=1000\begin{bmatrix}1\\1.1323\end{bmatrix},\\
V^1_1=\begin{bmatrix}1.1666&1.3209\end{bmatrix},\\
U^2_1=1000\begin{bmatrix}0\\1\end{bmatrix},\\
V^2_1=\begin{bmatrix}1.1666&-0.3290\end{bmatrix}.\\
\end{array}
$$

$$
A^1_2=0.5397J_2=U^1_2\otimes V^1_2=U^2_2\otimes V^2_2,
$$
where
$$
\begin{array}{l}
U^1_2=U^1_2=0.5397\begin{bmatrix}1\\1\end{bmatrix},\\
V^1_2=V^1_2=\begin{bmatrix}1&1\end{bmatrix}.\\
\end{array}
$$

$$
A^1_3=0.2294J_2=U^1_3\otimes V^1_3=U^2_3\otimes V^2_3,
$$
where
$$
\begin{array}{l}
U^1_3=U^2_3=0.2294\begin{bmatrix}1\\1\end{bmatrix},\\
V^1_3=V^2_3=\begin{bmatrix}1&1\end{bmatrix}.\\
\end{array}
$$

$$
A^1_4=0.4190J_2=U^1_4\otimes V^1_4=U^2_4\otimes V^2_4,
$$
 where
$$
\begin{array}{l}
U^1_4=U^2_4=0.4190\begin{bmatrix}1\\1\end{bmatrix},\\
V^1_4=V^2_4=\begin{bmatrix}1&1\end{bmatrix}.\\
\end{array}
$$
\item[(ii)]

$$
A^2_1=-72.74J_2=U^3_1\otimes V^3_1=U^4_1\otimes V^4_1, where
$$
$$
\begin{array}{l}
U^3_1=U^4_1=-72.741\begin{bmatrix}1\\1\end{bmatrix},\\
V^3_1=V^4_1=\begin{bmatrix}1&1\end{bmatrix}.\\
\end{array}
$$

$$
A^2_2=\begin{bmatrix}4.8642&-4.2920\\4.2920&-4.8642\end{bmatrix}=U^3_2\otimes V^3_2+U^4_2\otimes V^4_2,
$$
 where
$$
\begin{array}{l}
U^3_2=\begin{bmatrix}1\\0.8824\end{bmatrix},\\
V^3_2=\begin{bmatrix}4.8642&-4.2920\end{bmatrix},\\
U^4_2=\begin{bmatrix}0\\1\end{bmatrix},\\
V^4_2=\begin{bmatrix}0&-1.0771\end{bmatrix}.\\
\end{array}
$$

$$
A^2_3=0.2294J_2=U^5_2\otimes V^5_2=U^6_2\otimes V^6_2,
$$
 where
$$
\begin{array}{l}
U^5_2=U^6_2=0.2294\begin{bmatrix}1\\1\end{bmatrix},\\
V^5_2=V^6_2=\begin{bmatrix}1&1\end{bmatrix}.\\
\end{array}
$$

$$
A^2_4=0.4190J_2=U^7_2\otimes V^7_2=U^8_2\otimes V^8_2,
$$
 where
$$
\begin{array}{l}
U^7_2=U^8_2=0.4190\begin{bmatrix}1\\1\end{bmatrix},\\
V^7_2=V^8_2=\begin{bmatrix}1&1\end{bmatrix}.\\
\end{array}
$$

\item[(iii)]

$$A^3_1=10\begin{bmatrix}0.34&-0.34\\0.34&-0.34\end{bmatrix}=U^1_3\otimes V^1_3=U^{2}_3\otimes V^{2}_3,
$$
 where
$$
\begin{array}{l}
U^1_3=U^2_3=3.4\begin{bmatrix}1\\1\end{bmatrix},\\
V^1_3=V^2_3=\begin{bmatrix}1&-1\end{bmatrix},\\
\end{array}
$$

$$
A^3_2=9.9798J_2=U^{3}_3\otimes V^{3}_3=U^{4}_3\otimes V^{4}_3,
$$
 where
$$
\begin{array}{l}
U^{3}_3=U^4_3=9.9798\begin{bmatrix}1\\1\end{bmatrix},\\
V^{3}_3=V^4_3=\begin{bmatrix}1&1\end{bmatrix}.\\
\end{array}
$$

$$
A^3_3=\begin{bmatrix}-1.1957&1.0551\\-1.0551&1.1957\end{bmatrix}=U^{5}_3\otimes V^{5}_3+U^{6}_3\otimes V^{6}_3,
$$
 where
$$
\begin{array}{l}
U^{5}_3=\begin{bmatrix}1\\0.8824\end{bmatrix},\\
V^{5}_3=\begin{bmatrix}-1.1957&1.0551\end{bmatrix},\\
U^{6}_3=\begin{bmatrix}0\\1\end{bmatrix},\\
V^{6}_3=\begin{bmatrix}0&0.2647\end{bmatrix}.\\
\end{array}
$$

$$A^3_4=0.4190J_2=U^{7}_3\otimes V^{7}_3=U^{8}_3\otimes V^{8}_3,
$$
 where
$$
\begin{array}{l}
U^{7}_3=U^8_3=0.4190\begin{bmatrix}1\\1\end{bmatrix},\\
V^{8}_3=V^8_3=\begin{bmatrix}1&1\end{bmatrix}.\\
\end{array}
$$

\item[(iv)]

$$A^4_1=\begin{bmatrix}-0.8027&0.8027\\-0.8027&0.8027\end{bmatrix}=U^1_4\otimes V^1_4=U^2_4\otimes V^2_4,
$$
 where
$$
\begin{array}{l}
U^{1}_4=U^2_4=-0.8027\begin{bmatrix}1\\1\end{bmatrix},\\
V^{1}_4=V^2_4=\begin{bmatrix}1&-1\end{bmatrix}.\\
\end{array}
$$

$$A^4_2=1.0110J_2=U^3_4\otimes V^3_4=U^4_4\otimes V^4_4,
$$
 where
$$
\begin{array}{l}
U^{3}_4=U^4_4=1.0110\begin{bmatrix}1\\1\end{bmatrix},\\
V^{3}_4=V^4_4=\begin{bmatrix}1&1\end{bmatrix}.\\
\end{array}
$$

$$A^4_3=-5.9857J_2=U^5_4\otimes V^5_4=U^6_4\otimes V^6_4,
$$
 where
$$
\begin{array}{l}
U^{5}_4=U^6_4=-5.9857\begin{bmatrix}1\\1\end{bmatrix},\\
V^{5}_4=V^6_4=\begin{bmatrix}1&1\end{bmatrix}.\\
\end{array}
$$

$$A^4_4=\begin{bmatrix}-1.7498&1.5439\\-1.5439&1.7498\end{bmatrix}=U^7_4\otimes V^7_4+U^8_4\otimes V^8_4,
$$
 where
$$
\begin{array}{l}
U^{7}_4=\begin{bmatrix}1\\0.8823\end{bmatrix},\\
V^{7}_4=\begin{bmatrix}-1.7498&1.5439\end{bmatrix},\\
U^{8}_4=\begin{bmatrix}0\\1\end{bmatrix},\\
V^{8}_4=\begin{bmatrix}0&0.3876\end{bmatrix}.\\
\end{array}
$$
\end{itemize}

Finally, we have
$$
A=\dsum_{k=1}^8U^k_1\otimes V^k_1\otimes U^k_2\otimes V^k_2\otimes U^k_3\otimes V^k_3\otimes U^k_4\otimes V^k_4.
$$

\begin{rem}\label{r4.2.2}
\begin{itemize}
\item[(i)] It is obvious that when $d=2$, $\rank_t(A)=2$. According to the above computation we conclude that when $d=4$ $\rank_t(A)\leq 4$, when $d=8$ $\rank_t(A)\leq 8$.
Our conjecture is: when $d=4$ $\rank_t(A)=4$, when $d=8$ $\rank_t(A)=8$. The conjecture comes from a large number of numerical computations.

\item[(ii)] For finite sum decomposition, if the SVA is used, we may meet the local minimum stationary points.
There are two kinds of local minimum stationary points: (a) non-zero stationary point, in this case, the finite sum decomposition by using SVA successively can still going on; (b) zero stationary point, i.e., $\ltimes_{s=1}^dx_s^*=0$. Then the algorithm got stuck in an infinite loop.

When we use gradient-based algorithm for $d=8$ finite sum decomposition, we met case (b).

\item[(iii)] Comparing with singular-value based decomposition, the gradient-based decomposition has the following advantages: (1) Computational Complexity: The computational complexity of the gradient-based decomposition is $O(n)$, while the computational complexity of the singular-value based decomposition is at least  $2^{O(n)}$ (for precise solution) or $O(n^2log(1/\epsilon))$ (for approximate solution). (2) Computational Accuracy: From numerical examples it is easy to find that the gradient-based decomposition can provide much more accurate decomposition than singular-value based decomposition. (3) The SVA has no dimensional restriction, that is, it is applicable to ${\cal A}\in \R^{n_1\times \cdots \times n_d}$ with $n_i\neq n_j$, $i\neq j$. While the singular-value based decomposition needs each sub-blocks are square. An obvious disadvantage of SVA is for ill-conditioned data, the algorithm may got stuck in a stationary point which is not least square solution. Resetting initial values (might be many times) is then necessary.

\end{itemize}
\end{rem}

\section{Concludin Remarks}

The main contribution of this paper is to propose an efficient algorithm called the SVA, which is used for solving least square approximation of hypermatrices. It has some obvious advantages: (1) linear computational complexity; (2) high accuracy (3) no restriction on either order or dimension. Its main weakness is that the algorithm  may reach  a stationary point, which is only with a local minimum value, depending on the chosen initial values.
Since the computational complexity is very low, it is very likely that a Monte Carlo-like searching can find a least square stationary point.

Note that using SVA to solve the finite sum KPD problem does not ensure the sum number is minimum, i.e., it might be larger than the tensor rank.

In one word, the SVA is a very efficient algorithm for KPD of hypermatrices.

%
% \begin{IEEEbiography}[{\includegraphics[width=1in,height=1.25in,clip,keepaspectratio]{daizhancheng.png}}]
%	{Daizhan Cheng} graduated from Tsinghua University in 1970, received M.S. from Graduate School, Chinese Academy of Sciences in 1981, and Ph.D. from Washington University, St. Louis, in 1985. Since 1990, he is a professor with the Institute of Systems Science, AMSS, Chinese Academy of Sciences. He is the author/coauthor of 17 academic books, over 300 Journal Papers, and over 170 Conference Papers. He is IEEE Fellow (2006-), IFAC Fellow (2008-). He was a member of IEEE CSS Board of Governors (2009, 2015), and IFAC Council Member (2011-2014). He received the Second National Natural Science Award of China twice (in 2008 and 2014), the Outstanding Science and Technology Achievement Prize of CAS (2015), and the Automatica Best Paper Award (2008-2010), bestowed by IFAC.
%	He is the founder of the semi-tensor product of matrices.
%\end{IEEEbiography}
%


\begin{thebibliography}{00}
%
\bibitem{arz18} D.A. Arzanagh, G. Michailidis, Fast randomized algorithms for t-product based tensor operations and decompositions with applications to imaging data, {\it SIAM J. Imaging Sci.}, Vol. 11, No. 4, 2629-2664, 2018.
%
\bibitem{ayo24} B.A.M. Ayoub, J. Mitrovi\'{c}, M. Granitzer, KRONY-PT: GPT2 compressed with Kronecker products,
arXiv:2412.12351v1 [cs.LG] 16 Dec 2024.
%
\bibitem{bat17} K. Batselier, N. Wong, A constructive-degree Kronecker product decomposition of tensors, {\it Linear Algebra Appl.}, Vol. 24, e2097, 2017.
%
\bibitem{cai23} C. Cai, R. Chen, H. Xiao, Hybrid Kronecker product decomposition and approximation, {\it J. Comp. Graph. Stat.}, Vol. 32, No. 3, 838-852, 2023.
%
\bibitem{car70} J.D. Carroll, J. Chang, Analysis of individual differences in multidimansional scaling via an N-way generalization of ``Echart-Young" decomposition. {\it Psychometrika}, 35, 283-319, 1970.
%
\bibitem{cha00} F. Charles, C.F. Van Loan, The ubiquitous Kronecker product, {\it J. Comput. Appl. Math.}, Vol. 123, 85-100, 2000.
%
\bibitem{che12}  D. Cheng, H. Qi, Y. Zhao, {\it An Introduction to Semi-tensor Product of Matrices and Its Applications}, World Scientific, Singapore, 2012.
%
\bibitem{che19} D. Cheng, {\it From Dimension-Free Matrix Theory to Cross-Dimensional Dynamic Systems}, Elsevier, London, 2019.
%
\bibitem{chen24} C. Chen, {\it Tensor-Based Dynamical Systems - Theory and Applications},  Switzerland, Springer, 2024.
%
\bibitem{che25} D. Cheng, Z. Ji,  On Universal eigenvalues and eigenvectors of hypermatrices, {\it J. Franklin Inst.}, 362(17), 108126, 2025.
%
\bibitem{chepr} D. Cheng, A solution for Kronecker product decomposition, http://arXiv.org/abs/2509.22373, ({\it Sci. China, Inform. Sci.}, under revision.)
%
\bibitem{col62} A.R. Collar, On centrosymmetric and centroskew matrices, {\it Quart J. Mech. Appl. Math.}, 15(3), 265-281, 1962.
%
\bibitem{cos08} R. Costantini, L. Sbaiz, S. S\"{u}sstrunk, Higher order SVD analysis for dynamic texture synthesis, {\it IEEE Trans. Image Proces.}, 17(1), 2008.
%
\bibitem{dog22} L.M. Dogariu, J. Benesty, C. Paleologu, S. Ciochin\u{a}, Identification of room acoustic impulse responses via Kronecker product decompositions, {\it IEEE/ACM Trans. Audio, Speech, Language. Proc.}, Vol 30, 2828-2841, 2022.
%
\bibitem{eda21} A. Edalati, M. Tahaei, A. Rashid, V.P. Nia, J.J. Clark, M. Rezagholizadeh, Kronecker decomposition for GPT compression, arXiv:2110.08152v1 [cs.CL] 15 Oct 2021.
%
\bibitem{fen24} L. Feng, G. Yang, Deep Kronecker network, {\it Biometrika}, Vol. 111, No. 2, 707-714, 2024.
%
\bibitem{gar18} C. Garvey. C. Meng, J.G. Nagy, Singular value decomposition approximation via Kronecker symmations for imaging applications, {\it SIAM J. Matrix Anal. Appl.}, 39(4), 1836-1857, 2018.
%
\bibitem{gra10} L. Grasedyck, Hierarchical singular value decomposition of tensors, {\it SIAM. J. Matrix Anal. Appl.}, 31(4), 2029-2054, 2010.
%
\bibitem{ham22} M.G.A. Hameed, M.S. Tahaei, A. Mosleh, V.P. Nia, Convolutional neural network compression through generalized Kronecker product decomposition, {\it 36 AAAI Conf. on AI}, (AAAI-22), 771-779, 2022.
%
\bibitem{har70} R.A. Harshman, Foundations of the PARAFAC procedure: Model and conditions for an `explanatory' multi-mode factor analysis, {\it UCLA Workshop Papers} in phonetics, 16, 1-84, 1970.
%
\bibitem{joe24} S. Joel, S.K. Yadav, N.V. George, Adaptive low-rank DOA Estimation using complex Kronecker product decomposition, {\it IEEE Trans. Vehic. Tech.}, Vol. 73, No. 7, 2024.
%
\bibitem{kil13} M.E. Kilmer, K. Braman,  N. Hao, R.C. Hoover, Third-order tensors as operators on matrices: A theoretical and computational framework with applications in imaging, {\it SIAM J. Matrix Anal. Appl.}, Vol. 34, No. 1, 148-172, 2013.
%
\bibitem{lim13}		L. Lim, Tensors and Hypermatrices, in L. Hogben (Ed.) {\it Handbook of Linear Algebra} (2nd ed.), Chapter 15,
		Chapman and Hall/CRC.https://doi.org/10.1201/b16113, 2013.
%
\bibitem{mar04} C.D.M. Martin, Tensor decompositions workshp discussion notes: {\it American Institute of Mathematics}, http://www.aimath.org/WWN/tensordecomp/tensordecomp.pdf(accessed at August,2028),July 2004.
%
\bibitem{one25} A. Oneto, E. Ventura, Ranks of tensors: geometriy and applications, {\it Bollettino dell'Unione Matematica Italiana}, Vol. 18, 751-778, 2025.
%
\bibitem{pal18} C. Paleologu, J. Benesty, S. Ciochin\u{a}, Linear system identification based on a Kronecker product decomposition, {\it IEEE/ACM Trans.    Audio, Speech, Language. Proc.}, Vol 26, No. 10, 1793-1808, 2018.
%
\bibitem{tah21} M.S. Tahaci, E. Charlais, V.P. Nia, KroneckerBERT: Learning kronecker decomposition for pre-trained language models via knowledge distillation, arXiv:2109.06243v1, 2021.
%
\bibitem{van93} C.F. Van Loan, N. Pitsianis, Approximation with Kronecker product. In: M.S., Golub, G.H., De Moor, B.L.R. (eds) {\it Linear Algebra for Lare Scale and Real-TimeApplications}, NATO ASI Series, Vol. 232,  Springer, Dordrecht, 293-314, 1993.
%
\bibitem{van00}
C.F. van Loan, The ubiquitous Kronecker product,  {\it J. Comp. Appl. Math.}, 2000, 123, 85-100.
%
\bibitem{wan21} X. Wang, G. Huang, J. Benesty, J. Chen, I. Cohen, Time difference of arrival estimation based on a Kronecker product decomposition, {\it IEEE Sign. Proc. Lett.}, Vol. 28, 51-55, 2021.
%
\bibitem{wan22} Y. Wang, Y. Yang, Hot-SVD: higher order t-singular value decomposition for tensors based on tensor-tensor product, {\it Comput. Appl. Math.}, 41:394. 2022.
%
\bibitem{wei23} X. Wei, H. Li, G. Zhao, Kronecker product decomposition of Boolean matrix with application to topological structure analusis of Boolean networks, {\it Math. Model. Contr.}, Vol. 3, No. 4, 306-315, 2023.
%
\bibitem{wu23} Y. Wu, A new method of kronecker product decomposition, {\it Hindawi, J. Math.}, ID 9111626, 12 pages, 2023.
%
\bibitem{wu23b} S. Wu, L. Feng, Sparse Kronecker product decomposition: a general framework of signal region detection in image regression, {\it J. Royal Stat. Society Serier B: Statist. Methodol.}, Vol. 85, 783-809, 2023.
%
 \bibitem{zhu24} D. Zhu, Z. Zuo, M.M. Khalili, Training block-wise sparse models using Kronecker product decomposition, {\it 38th Conf. Neural Inform. Process. Sys.} NeurIPS, 1-12, 2024.
\end{thebibliography}
\end{document}